\documentclass[12pt]{article} 
\usepackage{latexsym,amsmath,amsfonts}
\usepackage{graphicx,color}
\makeatletter
  \newcommand\tinyv{\@setfontsize\tinyv{2pt}{3}}
\makeatother
\def\PP{{\mathbb P}}
\def\O{{\mathcal O}}

\def\M{{\mathcal M}}

\def\R{{\mathbb R}}
\def\C{{\mathbb C}}

\def\V{{\cal V}}
\def\U{{\cal U}}
\def\X{{\cal X}}
\def\T{{\mathbb T}}
\def\cT{{\cal T}}

\def\ell{{\it l}}

 \newtheorem{theorem}{Theorem}[section]
\newtheorem{lemma}[theorem]{Lemma} 
\newtheorem{prop}[theorem]{Proposition} 
\newtheorem{definition}[theorem]{Definition} 
\newtheorem{corollary}[theorem]{Corollary}
 
\newtheorem{remark}[theorem]{Remark}

\newtheorem{example}[theorem]{Example}  
 
\newcommand{\proof}{{\it Proof.}\ } 
\newcommand{\qed}{\hfill  $\Box$ }

\begin{document}
\title{A Geometric Perspective on the Singular Value Decomposition}
\author{Giorgio Ottaviani and Raffaella Paoletti}
\date{}

\maketitle

\begin{flushleft}{\it dedicated to Emilia Mezzetti}
\end{flushleft}

\begin{abstract}
This is an introductory survey, from a geometric perspective, on the Singular Value Decomposition (SVD) for real matrices, focusing on the role of the
Terracini Lemma. We extend this point of view to tensors, we define the singular space of a tensor as the space spanned by singular vector tuples
and we study some of its basic properties. 
\end{abstract}

\section{Introduction}
The {\it Singular Value Decomposition} (SVD) is a basic tool frequently used in Numerical Linear Algebra and in many applications, which generalizes the {\it Spectral Theorem} from symmetric $n\times n$ matrices to general $m\times n$ matrices. We introduce the reader to some of its beautiful properties, mainly related to the {\it Eckart-Young Theorem}, which has a geometric nature.
The implementation of a SVD algorithm in the computer algebra software {\it Macaulay2} allows a friendly use in many algebro-geometric computations.

This is the content of the paper. In Section 2 we see how the best rank $r$ approximation of a matrix can be described through its SVD; this is the celebrated Eckart-Young Theorem, that we revisit geometrically, thanks to the {\it Terracini Lemma}. In Section 3 we review the construction of the SVD of a matrix by means of the Spectral Theorem and we give a coordinate free version of SVD for linear maps between Euclidean vector spaces. In Section 4 we define the singular vector tuples of a tensor and we show how they are related to tensor rank; in the symmetric case, we get the eigentensors.  In Section 5 we define the singular space of a tensor, which is the space containing its singular vector tuples and we conclude with a discussion of the {\it Euclidean Distance (ED) degree}, introduced first in \cite{DHOST}. We thank the referee for many useful remarks.

\section{ SVD and the Eckart-Young theorem}\label{sec:svdey}

The vector space $\M = \M _{m,n}$ of $m\times n$ matrices with real entries has a natural filtration with subvarieties $\M _r=\{m\times n \textrm{\ matrices of rank\ }\le r\}$.
We have
$$\M_1\subset \M_2\subset\ldots \subset \M_{\min \{ m,n\}}$$
where the last subvariety $\M_{\min  \{ m,n\}}$ coincides with the ambient space.

\begin{theorem}[Singular Value Decomposition]\label{thm:svd} \hfill\break
Any real $m\times n$ matrix $A$ has the SVD 
$$A=U\Sigma V^t$$
where $U$, $V$ are orthogonal (respectively of size $m\times m$ and $n\times n$)
and \break
\noindent $\Sigma=\mathrm{Diag}(\sigma_1,\sigma_2,\ldots )$, with $\sigma_1\ge \sigma_2\ge\ldots \ge 0$. The $m\times n$ matrix $\Sigma$ has zero values at entries $(ij)$ with $i\neq j$ and sometimes it is called pseudodiagonal (we use the term diagonal only for square matrices).
\end{theorem}

The diagonal entries $\sigma_i$ are called {\it singular values} of $A$ and it is immediate to check that $\sigma_i^2$
are the eigenvalues of both the symmetric matrices $AA^t$ and $A^tA$. We give a proof of Theorem \ref{thm:svd} in \S \ref{sec:svdproof}. We recommend \cite{Ste}
for a nice historical survey about SVD.

\noindent Decomposing $\Sigma=\mathrm{Diag}(\sigma_1,0,0,\cdots)+\mathrm{Diag}(0,\sigma_2,0,\cdots)+\cdots=:\Sigma_1+\Sigma_2+\cdots$
we find
$$A=U\Sigma_1 V^t+U\Sigma_2 V^t+\cdots$$
and the maximum $i$ for which $\sigma _i \neq 0$ is equal to the rank of the matrix $A$.
\vskip0.5truecm
Denote by $u_k, v_l$ the columns, respectively, of $U$ and $V$ in the SVD above. From the equality $A=U\Sigma V^t$ we get $AV=U\Sigma $ and considering the $i$th columns we get 

$Av_i =(AV)_i =(U\Sigma )_i =\left( (u_1,\cdots ,u_m)\mathrm{Diag}(\sigma _1, \sigma _2, \cdots )\right) _i = \sigma _i u_i$, 

\noindent while, from the transposed equality $A^t=V\Sigma ^t U^t $, we get $A^t u_i=\sigma _iv_i$.

So if $1\le i\le min\{m,n\}$, the columns $u_i$ and $v_i$ satisfy the conditions
  \begin{equation}\label{eq:singpair}Av_i=\sigma_i u_i ~~\textrm {and}\ ~~A^tu_i=\sigma_i v_i.\end{equation}

\begin{definition}\label{def:sing} The pairs  $(u_i, v_i)$ in (\ref{eq:singpair}) are called  singular vector pairs.\hfill

More precisely, if  $1\le i\le min\{m,n\}$, the vectors $u_i ~ and ~ v_i$ are called, respectively, left-singular and right-singular vectors for the singular value $ \sigma _i$. \hfill 

If the value $\sigma _i$ appears only once in $\Sigma$, then the corresponding pair $(u_i,v_i)$ is unique up to sign multiplication.
\end{definition}
\begin{remark} The right-singular vectors corresponding to zero singular values of $A$ span the kernel of $A$; they are the last $n-rk(A)$ columns of $V$.

The left-singular vectors corresponding to non-zero singular values of $A$ span the image of $A$; they are the first $rk(A)$ columns of $U$.
\end{remark}

\begin{remark}
The uniqueness property mentioned in Definition \ref{def:sing} shows that SVD of a general matrix is unique up to
simultaneous sign change in each pair of singular vectors  $u_i$ and $v_i$. With an abuse of notation, 
 it is customary to think projectively and to refer to
{\it ``the''} SVD of $A$, forgetting the sign change. See Theorem \ref{thm:free} for more about uniqueness.
\end{remark}

For later use, we observe that $U\Sigma _i V^t = \sigma _i u_i\cdot v_i^t$.

 
\vskip0.5truecm

Let $|| -||$ denote the usual $l^2$ norm (called also Frobenius or Hilbert-Schmidt norm) on $\M $, that is $\forall A \in \M \hfill\break
||A||:=\sqrt {tr (AA^t)}=\sqrt {\sum _{i,j} a_{ij}^2}$.  
Note that if $A=U\Sigma V^t$, then $||A||=\sqrt {\sum _i \sigma _i^2}$. 
\vskip0.5truecm

The Eckart-Young Theorem uses SVD of the matrix $A$ to find the matrices in $\M_r$
which minimize the distance from $A$.



\begin{theorem}[Eckart-Young, 1936]\label{thm:ey1936}\hfill\noindent

Let $A=U\Sigma V^t$ be the SVD of a matrix $A$. Then
\begin{itemize}
\item{}$U\Sigma_1 V^t$ is the best rank $1$ approximation of $A$, that is\hfill\break
$||A-U\Sigma_1 V^t||\le ||A-X||$ for every  matrix $X$ of rank $1$.

\item{}For any $1\le r\le rank(A)$, $U\Sigma_1 V^t+\ldots +U\Sigma_r V^t$ is the best rank $r$ approximation of $A$, that is
$||A-U\Sigma_1 V^t-\ldots -U\Sigma_r V^t||\le ||A-X||$ for every  matrix $X$
of rank $\le r$.

\end{itemize}
\end{theorem}

Among the infinitely many rank one decompositions available for matrices,
the Eckart-Young Theorem detects the one which is particularly nice in optimization problems.
We will prove Theorem \ref{thm:ey1936} in the more general formulation of Theorem \ref{eyrevisited}.

\subsection{Secant varieties and the Terracini Lemma}
Secant varieties give basic geometric interpretation of rank of matrices and
also of rank of tensors, as we will see in section \ref{sec:tensors}.

Let $\X\subset\PP V$ be an irreducible variety. 
The $k$-secant variety of $\X$ is defined by
\begin{equation}\label{eq:sigmak}\sigma_k(\X):=\overline{\bigcup_{p_1,\ldots ,p_k\in \X}\PP{Span}\left\{p_1,\ldots , p_k\right\} }
\end{equation}
where $\PP{Span}\left\{p_1,\ldots , p_k\right\}$ is the smallest projective linear space containing $p_1,\ldots , p_k$  and the overbar means
Zariski closure (which is equivalent to Euclidean closure in all cases considered in this paper).

There is a filtration
$\X=\sigma_1(\X)\subset\sigma_2(\X)\subset\ldots$

This ascending chain stabilizes when it fills the ambient space.

\vskip 0.5cm

\begin{example}[Examples of secant varieties in matrix spaces.]
\label{exa:secmat}

We may identify the space $\M$ of $m\times n$ matrices with the tensor product $\R^m\otimes\R^n$.
Hence we have natural inclusions $\M_r\subset \R^m\otimes\R^n$. Since $\M_r$ are cones,
with an abuse of notation we may call with the same name the associated projective variety
$\M_r\subset \PP(\R^m\otimes\R^n)$.
The basic equality we need is
$$\sigma_r(\M_1)=\M_r$$
which corresponds to the fact that any rank $r$ matrix can be written as the sum of $r$ rank one matrices.

In this case the Zariski closure in (\ref{eq:sigmak}) is not necessary,
since the union is already closed.

\end{example}

The Terracini Lemma (see \cite{Lan} for a proof) describes the tangent space ${\mathbb T}$
of a $k$-secant variety at a general point.

\begin{lemma}[Terracini Lemma]\label{lem:terracini}\hfill\break
Let $z\in\PP{Span}\left\{p_1,\ldots ,p_k\right\}$ be general.
Then $${\mathbb T}_{z}\sigma_k(\X)=\PP{Span}\left\{\T_{p_1}\X,\ldots ,\T_{p_k}\X\right\}.$$
\end{lemma}

\begin{example}[Tangent spaces to $\M _r$]\label{tangent}

The tangent space to $\M _1$ at a point $u\otimes v$ is $\R ^m \otimes v +u\otimes \R ^n$:

any curve $\gamma(t)= u(t)\otimes v(t)$ in $\M _1$ with $\gamma (0)=u\otimes v$
has derivative for $ t = 0$  given by $u'(0)\otimes v+u\otimes v'(0)$ and since $u'(0), v'(0)$ are arbitrary vectors in $\R ^m, \R^n$ respectively, we get  the thesis.

As we have seen in Example \ref{exa:secmat}, the variety $\M_r$ can be identified with the $r$-secant variety of $\M _1$, so the tangent space to $\M_r$ at a point\hfill\break
 $U(\Sigma _1 +\cdots +\Sigma _r )V^t$ can be described, by the Terracini Lemma,  as \hfill\break
$\T_{U\Sigma _1 V^t}\M _1 +\cdots +\T_{U\Sigma _r V^t}\M _1=
\T_{\sigma _1 u_1\otimes v_1^t}\M _1+ \cdots +\T_{\sigma _r u_r\otimes v_r^t}\M_1=\hfill\break
 (\R ^m\otimes v_1^t +u_1\otimes \R ^n)+\cdots + (\R ^m\otimes v_r^t +u_r\otimes \R ^n).$

\end{example}

\vskip 0.5cm

\subsection{A geometric perspective on the Eckart-Young Theorem}

Consider the variety $\M_r\subset \R^m\otimes\R^n$ of matrices of rank $\le r$  and for any matrix $A\in \R^m\otimes\R^n$
let $d_A(-)=d(A,-)\colon \M_r\to\R$ be the (Euclidean) distance function 
from $A$. If $\mathrm{rk}A\ge r$ then the minimum  on $\M_r$ of $d_A$ is achieved on some matrices of rank $r$.
This can be proved by applying the following Theorem \ref{eyrevisited} to $\M_{r'}$ for any $r'\le r$.
Since the variety $\M _ r $ is singular
exactly on $\M_{r-1}$, the minimum of $d_A$ can be found among the critical points of $d_A$ on the smooth part $\M_r\setminus{\M_{r-1}}$.


\begin{theorem}[Eckart-Young revisited]\cite[Example 2.3]{DHOST}\label{eyrevisited}\hfill\break
Let $A=U\Sigma V^t$ be the SVD of a matrix $A$  and let $1\le r\le rk(A)$.
All the critical points  of the distance function from $A$ to the (smooth) variety  $\M _ r\setminus\M_{r-1}$ 
  are given by $U(\Sigma_{i_1}+\ldots+\Sigma_{i_r})V^t$,
where $ \Sigma_i=\mathrm{Diag}(0,\ldots,0,\sigma_i,0,\ldots,0)$, with $1\le i\le rk(A)$.
If the nonzero singular values of $A$ are distinct then  the number of critical points is
${rk(A)\choose r}$.
\end{theorem}

Note that $U\Sigma_iV^t$ are all the critical
points of the distance function from $A$ to the variety $\M_1$ of rank one matrices.
So we have the important fact that all the critical points of the distance function
from $A$ to $\M_1$ allow to recover the SVD of $A$.

For the proof of Theorem \ref{eyrevisited} we need

\begin{lemma}\label{lemma1}If $A_1=u_1\otimes v_1$,  $A_2=u_2\otimes v_2$ are two rank one matrices, then
$<A_1, A_2>=<u_1,u_2><v_1, v_2>$. 
\end{lemma}
\proof  $<A_1,A_2>=tr(A_1A_2^t) =\hfill\break
tr[\left( \left(\begin{array}{c}u_{11}\\ \vdots \\ u_{1m} \end{array} \right)\cdot (v_{11}, \cdots ,v_{1n})\right) 
\left( \left(\begin{array}{c}v_{21}\\ \vdots \\ v_{2n} \end{array}\right) \cdot (u_{21},\cdots ,u_{2m})\right)] =
\hfill\break
 \sum _i u_{1i}\left(\sum _k v_{1k}v_{2k}\right)u_{2i}=\sum _i u_{1i}u_{2i} \sum _k v_{1k}v_{2k}= <u_1,u_2><v_1,v_2>$.

\begin{lemma}\label{lemma2}
Let $B\in \M$.  If $<B,\R^m\otimes v>=0$, then $<\mathrm{Row}(B),v>=0$.

If $<B,u\otimes\R^n>=0$, then $<\mathrm{Col}(B), u>=0$.
\end{lemma}
\proof Let $\{ e_1,\cdots ,e_m\}$ be the canonical basis of $\R ^m$; then, by hypothesis,\hfill\break
 $<B, e_k\otimes v>=0 ~\forall k=1,\cdots ,m$. We have $0=tr\left[B (v^t\otimes e_k^t)\right] =\hfill\break
tr\left[B(0,\cdots ,0,v,0,\cdots ,0)\right] =
<B^k,v>$, where $B^k$ denotes the $k$th row of $B$, so that the space $\mathrm{Row}(B)$ is orthogonal to the vector $v$. \hfill\break
In a similar way, we get  $<\mathrm{Col}(B), u>=0$.
\vskip0.5truecm
By using Terracini Lemma \ref{lem:terracini} we can prove Theorem \ref{eyrevisited}.
\vskip0.5truecm

{\bf{Proof of Theorem \ref{eyrevisited}}}

The matrix $U (\Sigma _{i_1} + \cdots + \Sigma _{i_r}) V^t$ is a critical point of the distance function from $A$ to the variety $\M_{ r}$ if and only if the vector $A-(U (\Sigma _{i_1} + \cdots + \Sigma _{i_r})V^t)$ is orthogonal to the tangent space  [see \ref{tangent}]\hfill \break 
$ \T_{U (\Sigma _{i_1} + \cdots + \Sigma _{i_r}) V^t}\M _{ r}\, = 
 (\R ^m \otimes v_{i_1} ^t + u_{i_1} \otimes \R ^n) +\cdots +
(\R ^m \otimes v_{i_r} ^t + u_{i_r} \otimes \R ^n )$.

From the SVD of $A$ we have 
$A-(U (\Sigma _{i_1} + \cdots + \Sigma _{i_r})V^t) =
U(\Sigma _{j_1} + \cdots + \Sigma _{j_l})V^t=
\sigma _{j_1}u_{j_1}\otimes v_{j_1}^t +\, \cdots \, + \sigma _{j_l}u_{j_l}\otimes v_{j_l}^t $
 where $\{j_1, \cdots  , j_l\}$ is the set of indices given by the difference
$\{1, \cdots ,rk(A)\} \backslash \{i_1, \cdots , i_r\}$.

Let $\{ e_1, \cdots , e_m\} $ be the canonical basis of $\R ^m$. By Lemma $\ref{lemma1}$  we get:

$<\sigma _{j_h}u_{j_h}\otimes v_{j_h}^t, e_l\otimes v_{i_k}^t>=
 \sigma _{j_h} <u_{j_h},e_l><v_{j_h},v_{i_k}>=0$ since $v_{j_h},v_{i_k}$ are distinct columns of the orthogonal matrix $V$.  So the matrices $U\Sigma _{j_h}V^t$ are orthogonal to the spaces $\R ^m \otimes  v_{i_k}^t$.

 In a similar way, since $U$ is an orthogonal matrix,  the matrices $U\Sigma _{j_h}V^t$ are orthogonal to the spaces $  u_{i_k} \otimes \R ^n$. So  $A-(U (\Sigma _{i_1} + \cdots + \Sigma _{i_r})V^t)$ is orthogonal to the tangent space and $U (\Sigma _{i_1} + \cdots + \Sigma _{i_r})V^t$ is a critical point.

\def\niente1{Let now $B\in \M_{ r}$ be a critical point of the distance function from $A$ to $ \M_{ r}$. Then $A-B$ is orthogonal to the tangent space $\T_B \M _{ r}$. 

This implies by Lemma \ref{lemma2} that $Col(A-B)$ is orthogonal to $Col(B)$, hence $Col(B)\subset Col(A)$.
In the same way, again by Lemma \ref{lemma2}, we get $Row(B)\subset Row(A)$.
It is now straightforward to show that the critical point $B$ is of the desired type.
}

Let now $B\in \M_{ r}$ be a critical point of the distance function from $A$ to $ \M_{ r}$. Then $A-B$ is orthogonal to the tangent space $\T_B \M _{ r}$. 

Let $B=U' ( \Sigma '_1 +\cdots \Sigma '_r)V'^t, ~A-B= U'' ( \Sigma ''_1 +\cdots \Sigma ''_l)V''^t$ be SVD  of $B$ and $A-B$ respectively, with $\Sigma '_r \neq 0$ and $\Sigma ''_l \neq 0$.

Since $A-B$ is orthogonal to $\T_B \M_{ r} = 
 (\R^m \otimes {v'_1}^t + u'_1 \otimes \R ^n) +\cdots +
(\R ^m \otimes {v'_r} ^t + u'_r \otimes \R ^n )$, by Lemma $\ref{lemma2}$ we get
$<Col(A-B) , u'_k>=0 $ and $<Row(A-B), v'_k>=0~~ k=1,\cdots ,r$. In particular, $Col(A-B) $ is a vector subspace of $Span \{ u'_1, \cdots ,u'_r \}^{\perp }$ and has dimension at most $m-r$ while $Row(A-B)$ is a 
vector subspace of  $Span \{ v'_1, \cdots ,v'_r \}^{\perp }$ and has dimension at most $n-r$, so  that $l\le min\{m,n\}-r$.

From the equality 
$A-B= \left(u''_1,\ldots, u''_l,0\ldots,0\right) ( \Sigma ''_1 +\cdots \Sigma ''_l)V''^t$
we get $Col(A-B)\subset Span\{ u''_1,\ldots, u''_l\}$ and equality holds by dimensional reasons.

In a similar way, $Row(A-B)=Span \{ v''_1, \cdots , v''_l\}$. This implies that the orthonormal columns $u''_1, \cdots, u''_l, u'_1, \cdots ,u'_m$ can be completed with orthonormal $m-l-r$ columns of $\R ^m$ to obtain an orthogonal $m\times m$ matrix  $U$, while the orthonormal columns $v''_1, \cdots, v''_l,v'_1, \cdots ,v'_r$ can be completed with orthonormal $n-l-r$ columns of $\R ^n$ to obtain an orthogonal $n\times n$ matrix $V$.  

We get $A-B=U\left( \begin{array}{ccc}\Sigma '' &0&0\\ 0&0&0\\ 0&0&0\end{array} \right) V^t$, ~~
 $B=U\left( \begin{array}{ccc}0&0&0 \\ 0&\Sigma ' &0\\ 0&0&0\end{array} \right) V^t$, 
where \hfill\break
$\Sigma '' =\mathrm{Diag}(\sigma ''_1,\ldots , \sigma ''_l)$ and 
 $\Sigma ' =\mathrm{Diag}(\sigma '_1,\ldots , \sigma '_r)$. 

So $A=(A-B)+B=
U\left( \begin{array}{ccc}\Sigma ''&0&0 \\ 0&\Sigma ' &0\\ 0&0&0\end{array} \right) V^t$
can easily be transformed to a SVD of $A$ by just reordering the diagonal elements $\sigma '_i$'s and $\sigma ''_i$'s  and the critical point $B$ is of the desired type.\qed

\def\niente{
Let $B=U' ( \Sigma '_1 +\cdots \Sigma '_r)V'^t, ~A-B= U'' ( \Sigma ''_1 +\cdots \Sigma ''_l)V''^t$ be SVD of $B$ and $A-B$ respectively, with $\Sigma '_r \neq 0$ and $\Sigma ''_l \neq 0$.

Since $A-B$ is orthogonal to $\T_B \M_{ r} = 
 (\R^m \otimes {v'_1}^t + u'_1 \otimes \R ^n) +\cdots +
(\R ^m \otimes {v'_r} ^t + u'_r \otimes \R ^n )$, by Lemma $\ref{lemma2}$ we get
$<Col(A-B) , u'_k>=0 $ and $<Row(A-B), v'_k>=0~~ k=1,\cdots ,r$. In particular, $Col(A-B) $ is a vector subspace of $Span \{ u'_1, \cdots ,u'_r \}^{\perp }$ and has dimension at most $m-r$ while $Row(A-B)$ is a vector subspace of  $Span \{ v'_1, \cdots ,v'_r \}^{\perp }$ and has dimension at most $n-r$, so  that $l\le min\{m,n\}-r$.

The columns of the matrix $A-B$ are given by the product of the matrices 
$\left( \begin{array}{c}  u''_1,\cdots ,u''_m \end{array}\right)
\left( \begin{array}{cc} \mathrm {Diag}(\sigma ''_1, \cdots ,\sigma ''_l) &0 \\ 0&0 \end{array}\right)
\left( \begin{array}{c} ( v''_1)^t \\ \vdots \\ (v''_n)^t \end{array}\right) =\hfill\break
\left( \begin{array}{c}  u''_1,\cdots ,u''_m \end{array}\right)
\left(  \begin{array}{ccc} \sigma ''_1 (v''_1)_1 &\cdots & \sigma ''_1 (v''_1)_n\\ &\cdots & \\ 
 \sigma ''_l (v''_l)_1 &\cdots & \sigma ''_l (v''_l)_n \\ 0 &\cdots &0 \\ &\cdots & \\ 0&\cdots &0\end{array} \right)=
\left( \begin{array}{c}  u''_1,\cdots ,u''_l \end{array}\right)
\left(  \begin{array}{ccc} \sigma ''_1 (v''_1)_1 &\cdots & \sigma ''_1 (v''_1)_n\\ &\cdots & \\ 
 \sigma ''_l (v''_l)_1 &\cdots & \sigma ''_l (v''_l)_n \end{array} \right) $.
Since  $\sigma ''_1, \cdots , \sigma '' _l$ are nonzero and the vectors $v''_1, \cdots , v''_l$ are linearly independent, the second matrix has rank $=l$ and we have that $Col(A-B)=Span\{ u''_1, \cdots , u''_l\}$. 

In a similar way, $Row(A-B)=Span \{ v''_1, \cdots , v''_l\}$. This implies that the orthonormal columns $u''_1, \cdots, u''_l, u'_1, \cdots ,u'_r$ can be completed with orthonormal $m-l-r$ columns of $\R ^m$ to obtain an orthogonal $m\times m$ matrix  $U$ while the orthonormal columns $v''_1, \cdots, v''_l,v'_1, \cdots ,v'_r$ can be completed with orthonormal $n-l-r$ columns of $\R ^n$ to obtain an orthogonal $n\times n$ matrix $V$.  

We get $A-B=U\left( \begin{array}{ccc}\Sigma '' &0&0\\ 0&0&0\\ 0&0&0\end{array} \right) V^t$, ~~
 $B=U\left( \begin{array}{ccc}0&0&0 \\ 0&\Sigma ' &0\\ 0&0&0\end{array} \right) V^t$, 
where \hfill\break
$\Sigma '' =\mathrm{Diag}(\sigma ''_1,\ldots , \sigma ''_l)$ and 
 $\Sigma ' =\mathrm{Diag}(\sigma '_1,\ldots , \sigma '_r)$. 

So $A=(A-B)+B=
U\left( \begin{array}{ccc}\Sigma ''&0&0 \\ 0&\Sigma ' &0\\ 0&0&0\end{array} \right) V^t$
can be easily transformed in a SVD of $A$ just rearranging the diagonal elements $\sigma '_i$'s and $\sigma ''_i$'s in decreasing order and the critical point $B$ is of the desired type.}
\vskip 0.5cm

The following result has the same flavour of Eckart-Young Theorem \ref{eyrevisited}.

\begin{theorem}[Baaijens, Draisma]\cite[Theorem 3.2]{BD}\label{closestorthogonal}\hfill\break
Let $A=U\Sigma V^t$ be the SVD of a $n\times n$ matrix $A$. All the critical points  of the distance function from $A$ to the variety  $O(n)$
of orthogonal matrices are given by the  orthogonal matrices $U\mathrm{Diag}(\pm 1,\ldots, \pm 1)V^t$ and their number is
 $2^n$.
\end{theorem}

Actually, in \cite{BD}, the result is stated in a slightly different form, which is equivalent to this one, that we have chosen
to make more transparent the link with SVD. It is easy to check that, among the critical points computed in Theorem \ref{closestorthogonal}, the one with all plus signs, corresponding to the orthogonal matrix $UV^t$, gives the orthogonal matrix closest to $A$. This is called the L\"owdin orthogonalization (or symmetric orthogonalization) of $A$.

\section{SVD via the Spectral Theorem}\label{sec:svdproof}
In this section we prove Theorem \ref{thm:svd} as a consequence of the Spectral Theorem.
We recall
\begin{theorem}  [Spectral Theorem]\hfill\break
For any symmetric real matrix $B$, there exists an orthogonal matrix $V$ such that $V^{-1}BV=V^tBV$  is a diagonal matrix.
\end{theorem}
 
\begin{remark}\label{rmk:keraat}  Since the Euclidean inner product is positive definite, it is elementary to show that for any real $m\times n$ matrix $A$ we have
  $Ker (A^tA)=Ker (A)$ and $Ker (AA^t)=Ker (A^t)$.
\end{remark}

{\bf{Proof of Theorem \ref{thm:svd}}}\hfill\break
Let $A$ be an $m\times n$ matrix with real entries. The matrix $A^tA$ is a symmetric matrix of order $n$ and it's positive semidefinite.
By the Spectral Theorem, there exists an orthogonal matrix $V$ (of order $n$) such that
$$ V^{-1}(A^tA)V=V^t(A^tA)V=\begin{pmatrix} D &0 \\ 0&0 \end{pmatrix}$$
where $D$ is diagonal of order $r=\mathrm rk ( A^t A)=\mathrm rk (A)$ (see Remark \ref{rmk:keraat}) and is positive definite:
$D= \mathrm { Diag} (d_1, \cdots ,d_r)$  with  $d_1\ge d_2 \ge \cdots d_r >0$.

\noindent Let $v_1, \cdots ,v_n$ be the columns of $V$; then
$$ (A^tA)(v_1, \cdots ,v_n) =(v_1, \cdots ,v_n) \begin{pmatrix} D&0 \\ 0&0 \end{pmatrix} =
(d_1v_1, \cdots ,d_rv_r, 0, \cdots ,0)$$
and $v_{r+1}, \cdots ,v_n \, \in Ker (A^tA)=Ker (A)$ (see Remark \ref{rmk:keraat}).

\noindent Let $\sigma _i =\sqrt {d_i }, ~i=1, \cdots ,r$ and let $u_i=(1/\sigma _i)Av_i \in \R ^m$. These
vectors are orthonormal since  $ <u_i,u_j>=\frac 1 {\sigma _i \sigma _j}<Av_i,Av_j> = \frac 1 {\sigma _i \sigma _j}<v_i,A^tAv_j> = \break  \frac 1 {\sigma _i \sigma _j}<v_i,d_jv_j> = 
 \frac {\sigma _i} { \sigma _j}<v_i,v_j> =  \frac {\sigma _i} { \sigma _j}\delta _{ij}.$ 
 Thus it's possible to find $m-r$ orthonormal vectors in $\R ^m$ such that the matrix $U:=(u_1, \cdots ,u_r,u_{r+1}, \cdots ,u_m)$ is an $m\times m$ orthogonal matrix. Define $\Sigma := \begin{pmatrix}
D^{1/2} &0 \\ 0&0 \end{pmatrix}$ to be an $m\times n$ matrix with $m-r$ zero rows, $D^{1/2}=\mathrm {Diag} (\sigma _1, \cdots ,\sigma_r)$. Then
$$ U\Sigma V^t= \left( \frac 1{\sigma _1}Av_1, \cdots ,\frac 1{\sigma _r}Av_r, u_{r+1}, \cdots ,
u_m \right) \begin{pmatrix} \sigma _1 v_1^t \\ \vdots \\  \sigma _r v_r^t \\ 0\\ \vdots \\0 \end{pmatrix} =
A(v_1, \cdots ,v_r)\begin{pmatrix} v_1^t \\ \vdots \\ v_r^t \end{pmatrix}. $$

Since $V$ is orthogonal we have 
$$I_n=\begin{pmatrix}v_{1} &\cdots & v_n \end{pmatrix} \begin{pmatrix} v_{1}^t \\ \vdots
\\ v_n^t \end{pmatrix}=\begin{pmatrix}v_{1} &\cdots & v_r \end{pmatrix} \begin{pmatrix} v_{1}^t \\ \vdots
\\ v_r^t \end{pmatrix}+\begin{pmatrix}v_{r+1} &\cdots & v_n \end{pmatrix} \begin{pmatrix} v_{r+1}^t \\ \vdots
\\ v_n^t \end{pmatrix}$$

Hence we get 
$$ U\Sigma V^t=A\left( I_n - \begin{pmatrix}v_{r+1} &\cdots & v_n \end{pmatrix} \begin{pmatrix} v_{r+1}^t \\ \vdots
\\ v_n^t \end{pmatrix} \right) = A$$ 
since $v_{r+1}, \cdots ,v_n \, \in Ker(A)$.\hfill\qed

\begin{lemma} \label{lem:facts}
\noindent

\begin{itemize}
\item{(1)} Let $\sigma _1^2 >\cdots >\sigma _k ^2 >0$ be the distinct non zero eigenvalues of $A^tA$ and $V_i=Ker(A^tA-\sigma_i^2I_n) $ be the corresponding eigenspaces,

$V_0=Ker (A^tA)=Ker (A)$. Then
$$\R^n = \left(\oplus _{i=1}^k V_i\right) \oplus V_0$$    is an orthogonal decomposition of $ \R^n $.

\item{(2)} Let  $U_i=Ker(AA^t-\sigma_i^2I_m) $ ,  $U_0=Ker (AA^t)=Ker (A^t)$.
Then $$AV_i=U_i\textrm{\ \ and\ \ }A^tU_i=V_i\textrm{\ if\ }i=1, \cdots ,r.$$

\item{(3)} 
\hskip 4cm{$\R^m = \left(\oplus _{i=1}^k U_i\right) \oplus U_0$}\hfill
\vskip 0.2cm\noindent
 is   an orthogonal decomposition of $ \R^m $ and
$\sigma _1^2 >\cdots >\sigma _k ^2 >0$ are the distinct non zero eigenvalues of $AA^t$.

\item{(4)} The isomorphism 
$\frac 1{\sigma _i} A|_{V_i}: V_i \longrightarrow U_i$ is an isometry with inverse $\frac 1{\sigma _i} A^t |_{U_i}:U_i \longrightarrow V_i $.
\end{itemize}
\end{lemma}
\proof
(1) is the Spectral Theorem.
In order to prove (2), $AV_i \subseteq U_i$ since $\forall w\in V_i$ one has $(AA^t)(Aw)=A(A^tA)w =\sigma _i^2 Aw$. 
In a similar way, $A^tU_i\subseteq V_i$.  On the other hand, $\forall z\in U_i $ one has $ z= \frac 1 {\sigma _i^2}(AA^t)z = A(  \frac 1 {\sigma _i^2}A^tz) \in AV_i$ so that $AV_i=U_i$. In a similar way, $A^tU_i=V_i$.
(3) and (4) are immediate from (2). 
\vskip 0.5cm

Lemma \ref{lem:facts} may be interpretated as the following coordinate free version of SVD, that shows precisely in which sense SVD is unique.

\begin{theorem}[Coordinate free version of SVD]\label{thm:free}

\noindent

Let ${\V}$, ${\U}$ be real vector spaces of finite dimension endowed with inner products $<,>_{\V}$ and $<,>_{\U}$ and let $F\colon \V\to \U$ be a linear map with adjoint $F^t\colon \U\to \V$, defined by the property
$<Fv,u>_{\U}=<v,F^tu>_{\V}$ $\forall v\in \V, \forall u\in \U$.
Then there is a unique decomposition (SVD)
$$F=\sum_{i=1}^k\sigma_iF_i\qquad$$
 with $ \sigma_1> \ldots >\sigma_k> 0$, $F_i\colon \V\to \U$ linear maps such that
\begin{itemize}
\item{}$F_iF_j^t$ and $F_i^tF_j$ are both zero for any $i\neq j$, 
\item{}${F_i}_{|Im(F_i^t)}\colon Im(F_i^t)\to Im(F_i)$ is an isometry with inverse $F_i^t$.
\end{itemize}
Both the singular values $\sigma_i$ and
the linear maps $F_i$ are uniquely determined from $F$.

\end{theorem}

By taking the adjoint in Theorem \ref{thm:free}, $F^t=\sum_{i=1}^k\sigma_iF_i^t$ is the SVD of $F^t$.

The first interesting consequence is that
$$FF^t=\sum_{i=1}^k\sigma_i^2F_iF_i^t\textrm{\ \ \ and\ \ \ }
F^tF=\sum_{i=1}^k\sigma_i^2F_i^tF_i$$ are both spectral decomposition (and SVD) of the self-adjoint operators $FF^t$ and $F^tF$. This shows the uniqueness in Theorem \ref{thm:free}.
Note that $\V=\left(\oplus_{i=1}^k Im(F_i^t)\right)\bigoplus Ker F$
and $\U=\left(\oplus_{i=1}^k Im(F_i)\right)\bigoplus Ker F^t$ are both orthogonal decompositions and that $\mathrm{rk}F=\sum_{i=1}^k\mathrm{rk}F_i$.

Moreover, $F^+=\sum_{i=1}^k\sigma_i^{-1}F_i^t$ is the Moore-Penrose inverse of $F$,
expressing also the SVD of $F^+$.

Theorem \ref{thm:free} extends in a straightforward way to finite dimensional complex vector spaces $\V$ and $\U$ endowed with Hermitian inner products.

\section{Basics on tensors and tensor rank }\label{sec:tensors}
We consider tensors $A\in K^{n_1+1}\otimes\ldots\otimes K^{n_d+1}$
where  $K=\R$ or $\C$. It is convenient to consider complex tensors even if one is interested only in the real case. 


\begin{figure}[h]
\begin{center}
\includegraphics[width=30mm]{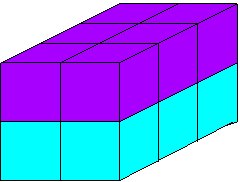}
\caption{The visualization of a tensor in $K^3\otimes K^2\otimes K^2$.}
\label{fig:322}
\end{center}
\end{figure}

Entries of $A$ are labelled by $d$ indices as $a_{i_1\ldots i_d}$.

For example,
the expression in coordinates of a $3\times 2\times 2$ tensor $A$ as in Figure \ref{fig:322} is,  with obvious notations,

\begin{equation*}
\begin{aligned}
A=a_{000}x_0y_0z_0+a_{001}x_0y_0z_1+a_{010}x_0y_1z_0+a_{011}x_0y_1z_1&+\\
a_{100}x_1y_0z_0+a_{101}x_1y_0z_1+a_{110}x_1y_1z_0+a_{111}x_1y_1z_1&+\\
a_{200}x_2y_0z_0+a_{201}x_2y_0z_1+a_{210}x_2y_1z_0+a_{211}x_2y_1z_1.
\end{aligned}
\end{equation*}

\begin{definition}A tensor $A$ is {\it decomposable} if there exist
$x^i\in K^{n_i+1}$, for $i=1,\ldots , d$,
such that
$a_{i_1\ldots i_d}=x^1_{i_1}x^2_{i_2}\ldots x^d_{i_d}$. In equivalent way,
$A=x^1\otimes\ldots\otimes x^d$.
\end{definition}


Define the rank of a tensor $A$ as the minimal number of decomposable summands
expressing $A$, that is
$$\mathrm{rk}(A):=\min\{r|A=\sum_{i=1}^rA_i, A_i\textrm{\ are decomposable}\}$$

For matrices, this coincides with usual rank.
For a (nonzero) tensor, decomposable  $\Longleftrightarrow$ rank one.

Any expression $A=\sum_{i=1}^rA_i$ with $A_i$ decomposable is called a {\it tensor decomposition}.

As for matrices, the space of tensors of format $(n_1+1)\times\ldots\times(n_d+1)$
has a natural filtration with subvarieties

$\cT_r=\{A\in K^{n_1+1}\otimes\ldots\otimes K^{n_d+1}|\textrm{\ rank }(A)\le r\}$.
We have
$$\cT_1\subset \cT_2\subset\ldots$$

Corrado Segre in XIX century understood this filtration
 in terms of projective geometry, since $\cT_i$ are cones.

The decomposable (or rank one) tensors give the ``Segre variety''

$$\begin{array}{ccc}
\cT_1\simeq\PP^{n_1}\times\ldots\times\PP^{n_d}&\subset&\PP(K^{n_1+1}\otimes\ldots\otimes K^{n_d+1})
\end{array}$$

The variety $\cT_k$ is again the $k$-secant variety of $\cT_1$, like in the case of matrices.

\vskip 0.5cm

For $K=\R$, the Euclidean inner product on each space $\R^{n_i+1}$ induces the inner product on the tensor product
$\R^{n_1+1}\otimes\ldots\otimes \R^{n_d+1}$ (compare with Lemma \ref{lemma1}).
With respect to this product we have the equality
$||x_1\otimes\ldots\otimes x_d||^2=\prod_{i=1}^d||x_i||^2$. 
A {\it best rank $r$ approximation} of a real tensor $A$ is a tensor in $\cT_r$ which minimizes the $l^2$-distance function from $A$. 
We will discuss mainly the best rank one approximations of $A$,
considering the  critical
points $T\in\cT_1$ for the $l^2$-distance function from $A$ to the variety $\cT_1$ of rank $1$ tensors,
trying to extend what we did in \S \ref{sec:svdey}. The condition 
that $T$ is a critical point is again that the tangent space at $T$ is orthogonal to the tensor $A-T$.

\vskip 0.5cm


\begin{theorem}[Lim, variational principle]\cite{Lim05}\label{thm:lim}

\noindent
The critical points $x_1\otimes\ldots\otimes x_d\in\cT_1$ of the distance function from $A\in\R^{n_1+1}\otimes\ldots\otimes\R^{n_d+1}$ 
to the variety $\cT_1$ of rank $1$ tensors are given by $d$-tuples
$(x_1,\ldots,x_d)\in \R^{n_1+1}\times\ldots\times\R^{n_d+1}$ such that
\begin{equation}\label{eq:singdple}
A\cdot(x_1\otimes\ldots\widehat{x_i}\ldots\otimes x_d)=\lambda x_i\quad\forall i=1,\ldots ,d\end{equation}
where $\lambda\in\R$, the dot  means contraction and the notation $\widehat{x_i}$ means that $x_i$ has been removed.
\end{theorem}

Note that the left-hand side of (\ref{eq:singdple}) is an element in the dual space of $\R^{n_i+1}$, so in order for (\ref{eq:singdple}) to be meaningful it is necessary to have a metric identifying $\R^{n_i+1}$
with its dual. We may normalize the factors $x_i$ of the tensor product $x_1\otimes\ldots\otimes x_d$ in such a way that $||x_i||^2$ does not depend on $i$.
Note that from (\ref{eq:singdple}) we get $A\cdot(x_1\otimes\ldots\otimes x_d)=\lambda ||x_i||^2$.
Here, $(x_1,\ldots,x_d)$ is called a {\it singular vector $d$-tuple} (defined independently by Qi in \cite{Qi07}) and $\lambda$ is called a {\it singular value}. 
Allowing complex solutions to  (\ref{eq:singdple}) , $\lambda$ may be complex.
\vskip 0.5cm

\begin{example}\label{exa:f}
We may compute all singular vector triples for the following tensor in $\R^3\otimes\R^3\otimes\R^2$
{\scriptsize\begin{equation*}\label{eq:f}
\begin{aligned}
f=&6x_0y_0z_0&+2x_1y_0z_0&+6x_2y_0z_0\\ &-2014x_0y_1z_0&+121x_1y_1z_0&-11x_2y_1z_0\\ &+48x_0y_2z_0&-13x_1y_2z_0&
-40x_2y_2z_0\\ &-31x_0y_0z_1&+93x_1y_0z_1&+97x_2y_0z_1\\ &+63x_0y_1z_1&+41x_1y_1z_1&-94x_2y_1z_1\\ &-3x_0y_2z_1&+47x_1y_2z_1&+4x_2y_2z_1  
\end{aligned}
\end{equation*}
}
We find $15$ singular vector triples, $9$ of them are real, $6$
of them make $3$ conjugate pairs.

The minimum distance is $184.038$
and the best rank one approximation is given by the singular vector triple

{\scriptsize $(x_0-.0595538x_1+.00358519x_2)(y_0-289.637y_1+6.98717y_2)(6.95378z_0-.2079687z_1)$.}
Tensor decomposition of $f$ can be computed from the Kronecker normal form
and gives $f$ as sum of three decomposable summands, that is

{\tiny
$f={(.450492 {x}_{0}-1.43768 {x}_{1}-1.40925 {x}_{2})(-.923877 {y}_{0}-.986098 {y}_{1}-.646584 {y}_{2})(.809777 {z}_{0}+68.2814 {z}_{1})}+\\
{(-.582772 {x}_{0}+.548689 {x}_{1}+1.93447 {x}_{2})(.148851 {y}_{0}-3.43755 {y}_{1}-1.07165 {y}_{2})(18.6866 {z}_{0}+28.1003 {z}_{1})}+\\
{(1.06175 {x}_{0}-.0802873 {x}_{1}-.0580488 {x}_{2})(-.0125305 {y}_{0}+3.22958 {y}_{1}-.0575754 {y}_{2})(-598.154 {z}_{0}+10.8017 {z}_{1})}$}

Note that the best rank one approximation  is unrelated to the three summands of minimal tensor decomposition,
in contrast with the Eckart-Young Theorem for matrices.

\end{example}

\begin{theorem}[Lim, variational principle in symmetric case]\cite{Lim05}\label{thm:lims}
The critical points of the distance function from $A\in\mathrm{Sym}^d\R^{n+1}$ 
to the variety $\cT_1$ of rank $1$ tensors are given by $d$-tuples
$x^d\in \mathrm{Sym}^d\R^{n+1}$ such that
\begin{equation}\label{eq:eigen}
A\cdot(x^{d-1})=\lambda x.\end{equation}
\end{theorem}

The tensor $x$ in (\ref{eq:eigen}) is called a {\it eigenvector}, the corresponding power $x^d$ is a {\it eigentensor} , $\lambda$ is called a {\it eigenvalue}.

\subsection{Dual varieties and hyperdeterminant}
If $\X\subset \PP V$
then $$\X^{*}:=\overline{\{H\in \PP V^{*}|\exists\textrm{\ smooth point\ }p\in \X\textrm{\ s.t.\ }
\T_p\X\subset H\}}$$
is called the {\it dual variety} of $\X$ (see \cite[Chapter 1]{GKZ}).
So $\X^{*}$ consists of hyperplanes tangent at some smooth point of $\X$.

In Euclidean setting, duality may be understood in terms of orthogonality.
Considering the affine cone of a projective variety $\X$,
the dual variety consists of the cone of all vectors which are orthogonal
to some tangent space to $\X$. 
\vskip 0.5cm

Let $m\le n$. The dual variety of $m\times n$  matrices of rank $\le r$ is given by 
$m\times n$ matrices of rank $\le m-r$ (\cite[Prop. 4.11]{GKZ}). In particular, the dual of the Segre variety
of matrices of rank $1$ is the determinant hypersurface.

Let $n_1\le\ldots\le n_d$. The dual variety of tensors of format $(n_1+1)\times\ldots\times (n_d+1)$
is by definition the {\it hyperdeterminant} hypersurface, whenever $n_d\le \sum_{i=1}^{d-1}n_i$.
Its equation is called the hyperdeterminant. Actually, this defines the hyperdeterminant up to scalar multiple, but
it can be normalized asking that the coefficient of its leading monomial is $1$.

\section{Basic properties of singular vector tuples and of eigentensors. The singular space of a tensor.}
\subsection{Counting the singular tuples}
In this subsection we expose
the results of \cite{FO} about the number of singular tuples (see Theorem \ref{thm:lim} ) of a general tensor.

\medskip
 

\begin{theorem}\cite{FO}\label{thm:count}
The number of (complex) singular $d$-tuples of a general tensor $t\in\PP(\R^{n_1+1}\otimes\ldots\otimes\R^{n_d+1})$
 is equal to the coefficient of  $\prod_{i=1}^dt_i^{n_i}$
in the polynomial
$$\prod_{i=1}^d\frac{{\hat{t_i}}^{n_i+1}-t_i^{n_i+1}}{\hat{t_i}-t_i}$$
where $\hat{t_i}=\sum_{j\neq i}t_j$.
\end{theorem}

\noindent Amazingly, for $d=2$ this formula gives the expected value
$\min(n_1+1, n_2+1)$.

For the proof, in \cite{FO} the $d$-tuples of singular vectors 
were expressed as zero loci of sections of a suitable
vector bundle on the Segre variety $\cT_1$.

Precisely,
let $\cT_1=\PP(\C^{n_1+1})\times\ldots\times \PP(\C^{n_d+1})$
and let $\pi_i\colon \cT_1\to \PP(\C^{n_i+1})$ be the projection on the $i$-th factor.
Let $\O(\underbrace{1,\ldots, 1}_d)$ be the very ample line bundle which gives
the Segre embedding and let $Q$ be the quotient bundle.
\medskip 

$\begin{array}{ccccccccccccc}\textrm{Then the bundle is\ }\oplus_{i=1}^d(\pi_i^*Q)&\otimes&\O(&1&,&\ldots&,&1&,&0&,1,\ldots, 1).\\
&&&&&&&&&&&\uparrow\\
&&&&&&&&&&&i\end{array}$

The top Chern class of this bundle gives the formula in Theorem \ref{thm:count}.
\medskip

In the format
$(\underbrace {2,\ldots , 2}_d)$ the number of singular $d$-tuples is $d!$. 

The following table lists the number of singular triples in the format $(d_1, d_2, d_3)$
\vskip 0.8cm

$\begin{array}{r|r|c}d_1, d_2, d_3&c(d_1,d_2,d_3)\\
\hline
 2, 2, 2& 6\\
 2, 2, n& 8&n\ge 3\\
 2, 3, 3& 15\\
 2, 3, n& 18&n\ge 4\\
 2, n, n&n(2n-1)\\
 3, 3, 3& 37\\
 3, 3, 4& 55\\
 3, 3, n& 61&n\ge 5\\
 3, 4, 4& 104\\
 3, 4, 5& 138\\
 3, 4, n& 148&n\ge 6\\
  \end{array}$

The number of singular $d$-tuples of a general tensor $A\in \C^{n_1+1}\otimes\ldots\otimes\C^{n_d+1}$, when $n_1,\ldots, n_{d-1}$ are fixed and $n_d$ increases,
  stabilizes for
$n_d\ge \sum_{i=1}^{d-1}n_i$, as it can be shown from Theorem \ref{thm:count}.

 For example, for a tensor of size
$2\times 2\times n$, there are $6$ singular vector triples for $ n = 2$
and $8$ singular vector triples for $n \ge 3$. 

The format with $n_d=\sum_{i=1}^{d-1}n_i$ is the  {\it boundary format}, well known in hyperdeterminant theory
\cite{GKZ}. It generalizes the square case for matrices.

The symmetric counterpart of Theorem \ref{thm:count} is the following

\begin{theorem}[Cartwright-Sturmfels]\cite{CS}\label{thm:counts}
The number of (complex) eigentensors of a general tensor $t\in\PP(\mathrm{Sym}^d\R^{n+1})$
 is equal to 
$$\frac{(d-1)^{n+1}-1}{d-2}.$$
\end{theorem}

\subsection{The singular space of a tensor}

We start informally to study the singular triples of a $3$-mode tensor $A$, later we will generalize to any tensor.
The singular triples $x\otimes y\otimes z$ of $A$ satisfy (see Theorem \ref{thm:lim} ) the equations

$$\sum_{i_0, i_1}A_{i_0i_1k}x_{i_0}y_{i_1}=\lambda z_{k}\quad\forall k$$

hence, by eliminating $\lambda$, the equations (for every $k<s$) $$\sum_{i_0, i_1}\left(A_{i_0i_1k}x_{i_0}y_{i_1}z_s-A_{i_0i_1s}x_{i_0}y_{i_1}z_k\right)=0$$
which are linear equations in the Segre embedding space.
These equations can be permuted on $x, y, z$ and give

\begin{equation}\label{eq:singspace}\left\{\begin{array}{cc}
\sum_{i_0, i_1}\left(A_{i_0i_1k}x_{i_0}y_{i_1}z_s-A_{i_0i_1s}x_{i_0}y_{i_1}z_k\right)=0&\textrm{\ for\ }0\le k<s\le n_3\\
\sum_{i_0, i_2}\left(A_{i_0 ki_2}x_{i_0}y_{s}z_{i_2}-A_{i_0 si_2}x_{i_0}y_{k}z_{i_2}\right)=0&\textrm{\ for\ }0\le k<s\le n_2\\
\sum_{i_1, i_2}\left(A_{k i_1 i_2}x_{s}y_{i_1}z_{i_2}-A_{s i_1 i_2}x_{k}y_{i_1}z_{i_2}\right)=0&\textrm{\ for\ }0\le k<s\le n_1\\
\end{array}\right.\end{equation}

These equations define the singular space of $A$, which is the linear span of all the singular vector triples of $A$.

The tensor $A$ belongs to the singular space of $A$, as it is trivially shown by the following identity (and its permutations)

$$\sum_{i_0,i_1}\left(A_{i_0i_1 k}A_{i_0i_1 s}-A_{i_0i_1 s}A_{i_0i_1 k}\right)=0.$$

In the symmetric case, the eigentensors $x^d$ of a symmetric tensor \break $A\in \mathrm{Sym}^d\C^{n+1}$ are defined by the
linear dependency of the two rows of the $2\times (n+1)$ matrix 

$$\begin{pmatrix}\nabla A(x^{d-1})\\ x\end{pmatrix}$$

Taking the $2\times 2$ minors we get the following

\begin{definition}
If $A\in \mathrm{Sym}^d\C^{n+1}$ is a symmetric tensor, then the singular space is given by the following ${{n+1}\choose 2}$ linear equations
in the unknowns $x^d$
$$\frac {\partial A(x^{d-1})}{\partial x_j}x_i-\frac {\partial A(x^{d-1})}{\partial x_i}x_j=0$$
\end{definition}

It follows from the definition that the singular space of $A$ is spanned
by all the eigentensors $x^d$ of $A$.

\begin{prop}\label{prop:singspacesym}

The symmetric tensor $A\in\mathrm{Sym}^d\C^{n+1}$ belongs to the singular space of $A$.
The dimension of the singular space is ${{n+d}\choose d}-{{n+1}\choose 2}$ .
The eigentensors are independent for a general $A$ (and then make a basis of the singular space) just in the cases
$\mathrm{Sym}^d\C^2$, $\mathrm{Sym}^2\C^{n+1}$, $\mathrm{Sym}^3\C^3$.
\end{prop}
\begin{proof}
To check that $A$ belongs to the singular space, consider dual variables $y_j=
\frac {\partial}{\partial x_j}$.
Then we have
$\left(\frac {\partial A}{\partial y_j}y_i-\frac {\partial A}{\partial y_i}y_j\right)\cdot A(x)=\frac {\partial A}{\partial y_j}\cdot\frac {\partial A}{\partial x_i}-
\frac {\partial A}{\partial y_i}\cdot\frac {\partial A}{\partial x_j}$,
which vanishes by symmetry.
To compute the dimension of the singular space, first recall that symmetric tensors in $\mathrm{Sym}^d\C^{n+1}$
correspond to homogeneous polynomials of degree $d$ in $n+1$ variables. We have to show that for a general polynomial $A$,
the ${{n+1}\choose 2}$ polynomials
$\frac {\partial A}{\partial x_j}x_i-\frac {\partial A}{\partial x_i}x_j$ for $i<j$ are independent. This is easily checked for the Fermat polynomial
$A=\sum_{i=0}^{n}x_i^d$ for $d\ge 3$ and for the polynomial
$A=\sum_{p<q}x_px_q$ for $d=2$.
The case listed are the ones where the inequality
$$\frac{(d-1)^{n+1}-1}{d-2}\ge {{n+d}\choose d}-{{n+1}\choose 2}$$
 is an equality (for $d=2$ the left-hand side reads as $\sum_{i=0}^n(d-1)^i=n+1$).
\end{proof}

Denote by $e_j$ the canonical basis in any vector space $\C^n$. 

\begin{prop}\label{prop:singspacegen}
Let $n_1\le\ldots\le n_d$ . If $A\in\C^{n_1+1}\otimes\ldots\otimes\C^{n_d+1}$ is a tensor, then the singular space of $A$ is given by the following $\sum_{i=1}^d{{n_i+1}\choose 2}$ linear equations.
The ${{n_i+1}\choose 2}$ equations of the $i$-th group ($i=1,\ldots, d$) for $x^1\otimes\ldots\otimes x^d$ are
$$\begin{array}{ccccccccccccccc}A(x^1,x^2,\ldots,&e_p&,\ldots,x^d)(x^i)_q-A(x^1,x^2,\ldots,&e_q&,\ldots,x^d)(x^i)_p=0\\
&\uparrow&&\uparrow\\
&i&&i\end{array}$$
for $0\le p<q\le n_i$.
The tensor $A$ belongs to this linear space, which we call again the singular space of $A$.
\end{prop}
\proof Let $A=\sum_{i_1,\ldots, i_d}A_{i_1,\ldots, i_d}x_{i_1}^1\otimes\ldots\otimes x_{i_d}^d$.
Then we have for the first group of equations
$$\sum_{i_2,\ldots, i_d} \left(A_{k,i_2,\ldots, i_d}A_{s,i_2,\ldots, i_d}-A_{s,i_2,\ldots, i_d}A_{k,i_2,\ldots, i_d}\right)=0\quad \textrm{for\ }
0\le k<s\le n_1$$ and the same argument works for the other groups of equations.
\qed

We state, with a sketch of the proof, the following generalization of the dimensional part of Prop. \ref{prop:singspacesym}.
\begin{prop}\label{prop:singspacedim}
Let $n_1\le\ldots\le n_d$ and $N=\prod_{i=1}^{d-1}(n_i+1)$. \break If $A\in\C^{n_1+1}\otimes\ldots\otimes\C^{n_d+1}$ is a tensor,
the dimension of the singular space of $A$ is
$$\left\{\begin{array}{lr}\prod_{i=1}^d(n_i+1)-\sum_{i=1}^{d}{{n_i+1}\choose 2}&\textrm{\ for\ }n_d+1\le N\\
\\
 {{N+1}\choose 2}-\sum_{i=1}^{d-1}{{n_i+1}\choose 2}&\textrm{\ for\ }n_d+1\ge N.\end{array}\right.$$
The singular $d$-tuples are independent (and then make a basis of this space) just in cases
$d=2$, $\C^2\otimes\C^2\otimes\C^2$, $\C^2\otimes\C^2\otimes\C^n$ for $n\ge 4$.
\end{prop}
\begin{proof}
Note that if $n_d+1\ge N$,
for any tensor $ A\in\C^{n_1+1}\otimes\ldots\otimes\C^{n_d+1}$,
there is a subspace $L\subset \C^{n_d+1}$ of dimension $N$
such that $ A\in\C^{n_1+1}\otimes\ldots\otimes\C^{n_{d-1}+1}\otimes L$,
hence all singular $d$-tuples of $A$ lie in $A\in\C^{n_1+1}\otimes\ldots\otimes\C^{n_{d-1}+1}\otimes L$.
Note that for $n_d+1=\prod_{i=1}^{d-1}(n_i+1)$, then the singular space has dimension
$N=\prod_{i=1}^d(n_i+1)-\sum_{i=1}^{d}{{n_i+1}\choose 2}$.
It can be shown that
posing $k(i_1,\ldots, i_{d-1})=\sum_{j=1}^{d-1}\left[\left(\prod_{s\le j-1}(n_s+1)\right)i_j\right]$,
then the tensor $A=\displaystyle\sum_{i_1=0}^{n_1}\ldots\sum_{i_{d-1}=0}^{n_{d-1}}k(i_1,\ldots, i_{d-1}) e^1_{i_1}\ldots e^{d-1}_{i_{d-1}}
e^d_{k(i_1,\ldots, i_{d-1})}$ is general in the sense that the $\sum_{i=1}^{d}{{n_i+1}\choose 2}$ 
corresponding equations 
are independent
(note that $k(i_1,\ldots, i_{d-1})$ covers all integers between $0$ and $N-1$).
For $n_d+1\ge N$ the dimension stabilizes to
$N^2-\sum_{i=1}^{d-1}{{n_i+1}\choose 2}-{{N+1}\choose 2}={{N+1}\choose 2}-\sum_{i=1}^{d-1}{{n_i+1}\choose 2}$.
\end{proof}

\begin{remark} For a general $A\in\C^{n_1+1}\otimes\ldots\otimes\C^{n_d+1}$,
the $\sum_{i=1}^{d}{{n_i+1}\choose 2}$ linear equations of the singular space of $A$
are independent if $n_d+1\le N+1$. 
\end{remark}

\begin{remark}
In the case of symmetric matrices, the singular space of $A$ consists of
all matrices commuting with $A$. If $A$ is regular, this space is spanned by the powers of $A$. If $A$ is any matrix (not necessarily symmetric), the singular space of $A$ consists of all matrices with the same singular vector pairs as $A$.
These properties seem not to generalize to arbitrary tensors. Indeed the tensors in the singular space of a tensor $A$ may have singular vectors different from those of $A$, even in the symmetric case. This is apparent for binary forms.
The polynomials $g$ having the same eigentensors as $f$, satisfy the equation
$g_xy-g_yx=\lambda(f_xy-f_yx)$ for some $\lambda$,
which in degree $d$ even has (in general) the solutions $g=\mu_1f+\mu_2(x^2+y^2)^{d/2}$ with $\mu_1, \mu_2\in\C$,
while  for degree $d$ odd has (in general) the solutions $g=\mu_1f$. In both cases, these solutions are strictly contained
in the singular space of $f$.
\end{remark}

In any case, a positive result which follows from Prop. \ref{prop:singspacesym},
\ref{prop:singspacegen}, \ref{prop:singspacedim} is the following
\begin{corollary}\hfill
\begin{itemize}
\item{(i)} Let $n_1\le\ldots\le n_d$, $N=\prod_{i=1}^{d-1}(n_i+1)$,
$M=\min(N,n_d+1)$.
A general tensor $A\in \C^{n_1+1}\otimes\ldots\otimes\C^{n_d+1}$ has a tensor decomposition given by $NM-\sum_{i=1}^{d-1}{{n_i+1}\choose 2}-{{N+1}\choose 2}$
 singular vector $d$-tuples.

\item{(ii)} A general symmetric tensor $A\in \mathrm{Sym}^d\C^{n+1}$ has a symmetric tensor decomposition given by ${{n+d}\choose d}-{{n+1}\choose 2}$ 
eigentensors.
\end{itemize}

The decomposition in (i) is not minimal unless $d=2$, when it is given by the SVD.

The decomposition in (ii) is not minimal unless $d=2$, when it is the spectral decomposition, as sum of $(n+1)$ (squared) eigenvectors.

\end{corollary}

\subsection{The Euclidean Distance Degree and its duality property}
The construction of critical points of the distance from a point $p$,
can be generalized to any affine (real) algebraic variety $\X$.

Following \cite{DHOST}, we call Euclidean Distance Degree (shortly {\it ED degree}) the number of critical points
of $d_p=d(p,-)\colon \X\to\R$, allowing complex solutions. As before, the number of critical points does not depend on $p$, provided $p$ is generic.
For a elementary introduction, see the nice survey \cite{Thomas}.

Theorem \ref{eyrevisited} says that the ED degree of the variety $\M_r$ defined in
\S \ref{sec:svdey} is ${{\min  \{ m,n\}}\choose r}$, while
Theorem \ref{closestorthogonal} says that the ED degree of the variety $O(n)$ is $2^n$. 
The values computed in Theorem \ref{thm:count} give the ED degree of the Segre variety $\PP^{n_1}\times\ldots\times\PP^{n_d}$, while the 
Cartwright-Sturmfels formula in Theorem \ref{thm:counts}
gives the ED degree of the Veronese variety $v_d(\PP^n)$.

\begin{theorem}\cite[Theorem 5.2, Corollary 8.3]{DHOST}\label{thm:critrank} Let $p$ be a tensor. There is a canonical bijection between 
\begin{itemize}
\item{}critical points of the distance
from $p$ to rank $\le 1$ tensors

\item{}critical points of the distance from $p$ to hyperdeterminant hypersurface.
\end{itemize}
Correspondence is $x\mapsto p-x $
\end{theorem}

In particular, from the $15$ critical points for the distance from the $3\times 3\times 2$
tensor $f$ defined in Example \ref{exa:f} to the variety of rank one matrices, we may recover
the $15$ critical points for the distance from $f$  to the
hyperdeterminant hypersurface. It follows that $\mathrm{Det}(f-p_i)=0$ for the $15$ critical points $p_i$.

\noindent\hskip 0.3cm The following result generalizes Theorem \ref{thm:critrank} to any projective variety $\X$.

\begin{theorem}\cite[Theorem 5.2]{DHOST} Let $\X\subset\PP^n$ be a projective variety, $p\in\PP^n$. There is a canonical bijection between 
\begin{itemize}
\item{}critical points of the distance
from $p$ to $\X$

\item{}critical points of the distance from $p$ to
the dual variety $\X^*$.
\end{itemize}
Correspondence is $x\mapsto p-x $. In particular
$\mathrm{EDdegree} (\X)=\mathrm{EDdegree }(\X^*)$
\end{theorem}

\begin{figure}[h]
\begin{center}
\includegraphics[scale=.6]{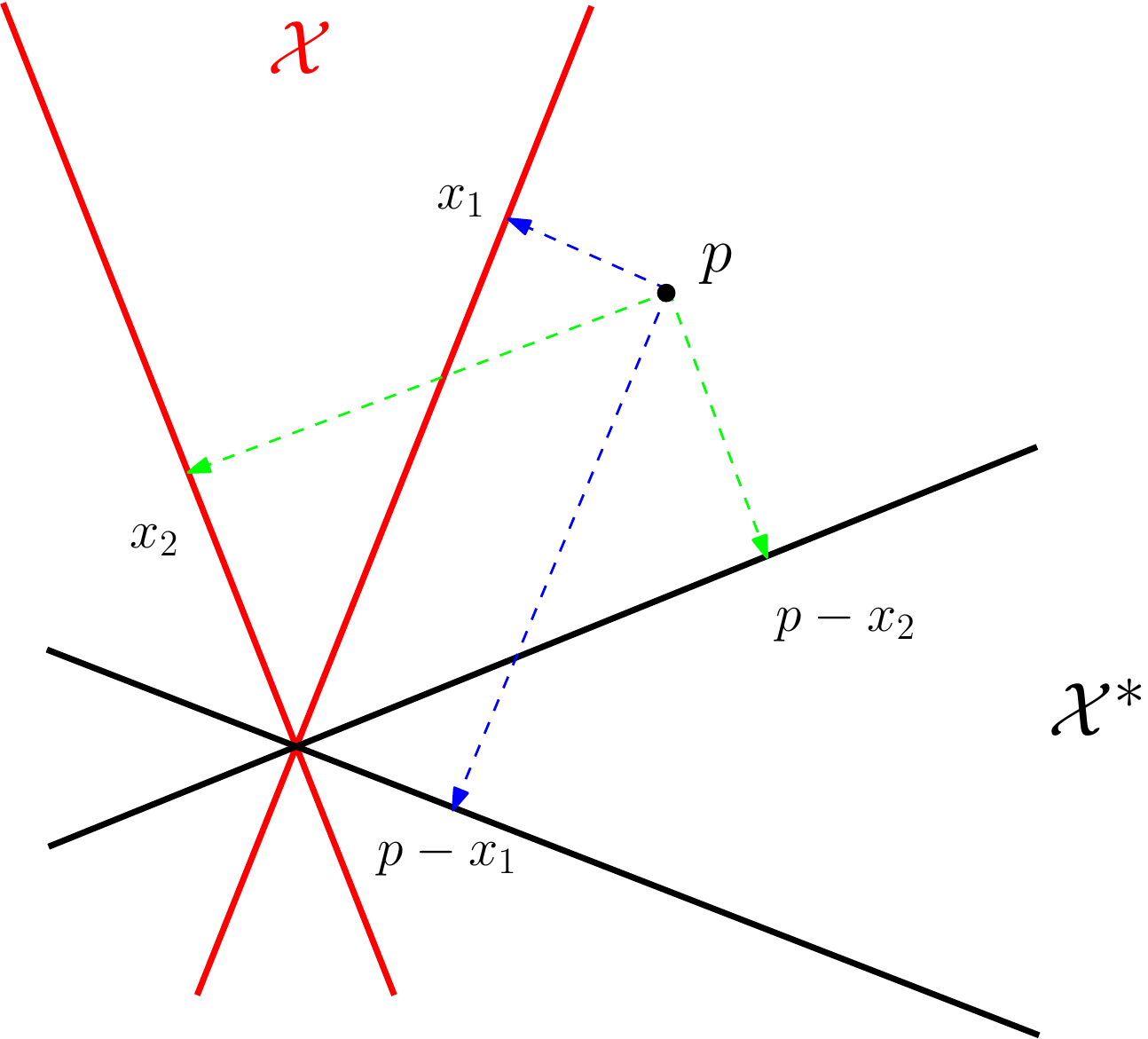} \,\,\,
\caption{The bijection between critical points on $\X$ and
critical points on $\X^*$.}
\label{fig:dualED}
\end{center}
\end{figure}

\subsection{Higher order SVD}
In \cite{HOSVD}, L. De Lathauwer, B. De Moor, and J. Vandewalle proposed
a higher order generalization of SVD. This paper has been quite influential
and we sketch this contruction for completeness (in the complex field).

\begin{theorem}[HOSVD,  De Lathauwer, De Moor, Vandewalle, \cite{HOSVD}]
A tensor $A\in\C^{n_1+1}\otimes\ldots\otimes\C^{n_d+1}$ can be multiplied in the $i$-th mode by unitary matrices $U_i\in U(n_i+1)$ in such a way that the resulting tensor $S$ has the following properties:
\begin{enumerate}
\item{(i)} (all-orthogonality) For any $i=1,\ldots, d$ and $\alpha=0,\ldots, n_i$
denote by $S^i_{\alpha}$ the slice 
in $\C^{n_1+1}\otimes\ldots\widehat{\C^{n_i+1}}\ldots\otimes\C^{n_d+1}$ obtained by fixing
the $i$-index equal to $\alpha$. Then for $0\le\alpha<\beta\le n_i$
we have $\overline{S^i_{\alpha}}\cdot S^i_{\beta}=0$, that is any two parallel slices are orthogonal
according to Hermitian product.
\item{(ii)} (ordering) for the Hermitian norm, for all $i=1,\ldots, d$ 
$$\left\lVert S^i_{0}\right\rVert\ge\left\lVert S^i_{1}\right\rVert\ge\ldots\ge
\left\lVert S^i_{n_i}\right\rVert$$
\end{enumerate}
\end{theorem}
$\left\lVert S^i_{j}\right\rVert$ are the $i$-mode singular values
and the columns of $U_i$ are the $i$-mode singular vectors.
For $d=2$, $\left\lVert S^i_{j}\right\rVert$ do not depend on $i$
and we get the classical SVD. This notion has an efficient algorithm computing it.
We do not pursue it further because the link with the 
critical points of the distance is weak, although it can be employed by suitable iterating methods

 \noindent
  \textsc{G. Ottaviani, R. Paoletti} -
  Dipartimento di Matematica e Informatica ``U. Dini'', Universit\`a di Firenze, viale Morgagni 67/A, 50134 Firenze (Italy). e-mail: \texttt{ottavian@math.unifi.it, raffy@math.unifi.it}
  
 \end{document}